\documentclass[12pt]{amsart}

\usepackage{pb-diagram} 

\usepackage{pdfsync}
\usepackage{xcolor}

\usepackage{geometry}
\usepackage{lscape}
\usepackage{graphicx}
\usepackage{epstopdf}    
\usepackage[matrix,arrow,curve,frame]{xy}
\usepackage{amsmath,amsthm,amssymb,enumerate}
\usepackage{latexsym}
\usepackage{amscd}
\usepackage{amssymb}
\usepackage{comment}

\usepackage{color}

\numberwithin{equation}{section}

\theoremstyle{plain}  
\newtheorem{thm}[equation]{Theorem}

\newtheorem{lemma}[equation]{Lemma}

\theoremstyle{definition}  
\newtheorem{defn}[equation]{Definition}

\newtheorem{remark}[equation]{Remark}

\newtheorem{conditions}[equation]{Conditions}

 \newcommand{\Smash} {\wedge}

 \DeclareRobustCommand{\bigWedge} {\bigvee}
 \newcommand{\B} {\field B}

\newcommand{\lra}{\longrightarrow}

\newcommand{\Z}{\mathbb Z}
\newcommand{\C}{\mathbb C}

\newcommand{\D}{\mathbb P}

\newcommand{\MU}{\mathbb {MU}}
\newcommand{\E}{\mathbb E}

\newcommand{\BP}{\mathbb{BP}}

\newcommand{\BPRm}[2]{\underline{\B\D}\langle #1 \rangle_{#2}}
\newcommand{\BPR}[1]{\BPRm{#1}{(2^{#1+1}-1)(1+\alpha)}}

\DeclareMathOperator{\EM}{\Ek(\Z,2m+1)}
\DeclareMathOperator{\EMF}{\Ek(\Z/2^q,2m)}
\DeclareMathOperator{\EMf}{\Ek(\Z/2,m)}
\DeclareMathOperator{\Ek}{K}

\DeclareMathOperator{\BOg}{BO}
\DeclareMathOperator{\BSg}{BSpin}
\DeclareMathOperator{\BSOg}{BSO}
\DeclareMathOperator{\BUg}{BU}
\DeclareMathOperator{\BSUg}{BSU}

\DeclareMathOperator{\id}{id}

\DeclareMathOperator{\Maps}{Maps}

 \newcommand{\field}[1] {\mathbb #1}
 \newcommand{\lift}[1]{#1'}

\begin{document}
\pagestyle{plain}

\title
{Landweber flat real pairs and $ER(n)$-cohomology.}
\author{Nitu Kitchloo}
\address{Department of Mathematics, Johns Hopkins University, Baltimore, USA}
\email{nitu@math.jhu.edu}
\author{Vitaly Lorman}
\address{Department of Mathematics, University of Rochester, Rochester, USA}
\email{vlorman@ur.rochester.edu}
\author{W. Stephen Wilson}
\address{Department of Mathematics, Johns Hopkins University, Baltimore, USA}
\email{wsw@math.jhu.edu}
\thanks{The authors would like to thank the referee for a careful reading and helpful suggestions.}
\thanks{The first author is supported in part by the NSF through grant DMS
  1307875.}
  \thanks{The third author would like to thank the Academy of Mathematics and Systems Science, part of the Chinese Academy of Sciences, for its hospitality and support during part of the research for this paper.}

\date{\today}


{\abstract

We take advantage of the internal algebraic structure of the Bockstein spectral sequence converging to $ER(n)^*(pt)$ to prove that for spaces $Z$ that are part of Landweber flat real pairs with respect to $E(n)$ (see Definition \ref{realpair}), the cohomology ring $ER(n)^*(Z)$ can be obtained from $E(n)^*(Z)$ by base change. In particular, our results allow us to compute the Real Johnson-Wilson cohomology of the Eilenberg-MacLane spaces $Z = \EM, \EMF, \EMf$ for any natural numbers $m$ and $q$, as well as connective covers of $\BOg$: $\BOg, \BSOg, \BSg$ and $\BOg\langle 8 \rangle$.}
\maketitle

\tableofcontents

\section{Introduction}

\medskip
The Real Johnson-Wilson theories, $ER(n)$ for $n >0$ were first introduced in \cite{HK} and further developed by the first and third authors in \cite{Nitufib}. They are a family of cohomology theories generalizing real $K$-theory, $KO$ (which, 2-locally, is $ER(1)$). They are not complex-oriented, and their coefficients contain 2-torsion. They are formed by taking fixed points of a $\Z/2$-action on the complex-oriented Johnson-Wilson theories, $E(n)$. As such, the Real Johnson-Wilson theories have proved to be remarkably amenable to computations. For example, their properties were exploited in \cite{NituP, NituP2, Ban} to demonstrate some new nonimmersions of real projective spaces.

\medskip
\noindent
The goal of this paper is to describe a class of spaces, which includes certain Eilenberg-MacLane spaces and connective covers of $\BOg$, whose $ER(n)$-cohomology may be computed from $E(n)$-cohomology entirely formally by means of base change. This paper adds to a growing list of spaces whose $ER(n)$-cohomology is now known: \cite{NituP, NituP2, Ban, KW2, Lor16, KLW16a} all provide further examples. 

\medskip
\noindent
The results of this paper are far more general than those in previous examples. As in many $ER(n)$-computations, our main tool is the Bockstein spectral sequence (BSS), developed by the first and third authors in \cite{NituP, NituP2}, which computes $ER(n)$-cohomology from $E(n)$-cohomology. However, our methods  bypass the finesse of seriously engaging with BSS differentials that is required in other examples. Instead, we develop the framework of Landweber flat real pairs (Definition \ref{realpair}), to which all of the spaces of interest in this paper belong. This allows us to systematically produce permanent cycles in the BSS in such a way that all differentials play out on the coefficients (in a suitable sense) leading to a simple base change formula.

\medskip
\noindent
The statement of our general result requires an extensive 
collection of highly technical definitions.  However, the statement
for our examples only requires one straightforward, but exotic,
definition.

Recall the coefficients for 
the $p=2$ Johnson-Wilson theory $E(n)$, 
$$
E(n)^*=\Z_{(2)}[v_1, v_2, \dots, v_{n-1}, v_n^{\pm 1}], 
\text{ with }|v_i|=-2(2^i-1).
$$

\medskip
\begin{defn}\label{defhat}
Let $\lambda=\lambda(n) = 2^{n+2}(2^{n-1}-1)+1$. 
For $Z$ a space with $E(n)^*(Z)$ concentrated in even degrees, we define an endomorphism of $E(n)^*(Z)$ as follows. If $z \in E(n)^{2k}(Z)$, we send $z$ to $\hat{z}:=zv_n^{k(2^n-1)}\in E(n)^{k(1-\lambda)}(Z)$. Note that this is a map of rings (but not graded rings), and it is in fact injective for $n>1$ since $v_n$ is a unit and degrees scale by a nonzero factor of $(1-\lambda)$ (see the following remark for the anomalous case of $n=1$). We denote its image by $\hat{E}(n)^*(Z)$. It is thus a subalgebra of $E(n)^*(Z)$ and is isomorphic to $E(n)^*(Z)$. Note that the use of $\hat{E}(n)^*(-)$ here does not denote completion; we hope this will cause no confusion.

In the case where $Z$ is a point, there is a canonical lift of $\hat{E}(n)^* \subset E(n)^*$ to a subalgebra of $ER(n)^*$ along the inclusion of fixed points map $ER(n)^* \longrightarrow E(n)^*$. In particular, $\hat{v}_k$ maps to $v_k v_n^{-(2^n-1)(2^k-1)}$ in $E(n)^*$.\end{defn}

\begin{remark} \label{anomaly} 
As observed in \cite{KW2}, the case $n=1$ is anomalous since $1-\lambda = 0$. Indeed, the image of $E(1)^*(Z)$ under the above (hat) map is the entire subalgebra $E(1)^0(Z)$ given by elements in homogeneous degree zero and the map is not injective in this case. Nevertheless, we define $\hat{E}(1)^*(Z)$ to be $E(1)^0 (Z)$ for the rest of the paper and our arguments will go through with only minor changes as indicated.
\end{remark}

We can now state the main applications of our general result. Keep in mind that $E(n)^*(Z)$ is, in principle, known for all of the examples below (see Remark \ref{E(n)}). The space $\widetilde{\BSg}$ denotes the fiber of the map $p_1: \BSg \longrightarrow \Ek(\Z, 4)$ (see also \cite[Proof of Theorem 1.13]{KLW}).
\medskip

\begin{thm}\label{main3}
Let $Z$ be any of the spaces $\EM$, $\EMF, \EMf$ for 
arbitrary natural numbers $m$ and $q$, or the connective 
covers $\BOg, \BSOg, \BSg, \widetilde{\BSg}$. 
Then the subalgebra $\hat{E}(n)^*(Z)$ of $E(n)^*(Z)$ lifts to a subalgebra of $ER(n)^*(Z)$ which we denote by the same name and which is formally isomorphic to $E(n)^*(Z)$ (see Remark \ref{anomaly} for $n=1$). Furthermore, the canonical map
\[ ER(n)^* \otimes_{\hat{E}(n)^*} \hat{E}(n)^*(Z) \longrightarrow ER(n)^*(Z), \]
is an isomorphism of graded algebras for all $n$. 
\end{thm}

\bigskip
\begin{remark}\label{E(n)}
For the spaces $Z$ above, much is known about their $E(n)$-cohomology.  For example, for the $\Ek(\Z,m)$, the Morava K-theory is known from \cite{RW}.  Then \cite{RWY} goes on to show that this implies $BP^*(\Ek(\Z,m))$ and $E(n)^*(\Ek(\Z,m))$ are Landweber flat.  From this it is straightforward to compute the $E(n)$-cohomology completed at the ideal $I=(2, v_1, \dots,v_{n-1})$.  Things are more complicated without completion.  Following \cite{RWY}, there are spaces and maps $\Ek(\Z,m) \longrightarrow X_m \longrightarrow Y_m$ such that $E(n)^*(\Ek(\Z,m))$ is the algebra quotient of $E(n)^*(X_m)$ by the image of $E(n)^*(Y_m)$, where $E(n)^*(X_m)$ and $E(n)^*(Y_m)$ are both power series rings on generators of known degrees (see \cite{KLW} for the analogous sequences for connective covers of $\BOg$).  The spaces $X_m$ and $Y_m$ are torsion-free spaces with well understood polynomial cohomology.  In principle, the map is computable.  Practice is another thing (see \cite[Section 8.5]{RWY}).  The difficulties involved with any attempt to describe $E(n)^*(\Ek(\Z,m))$ more explicitly were what led to the definition of Landweber flat real pairs below (Definition \ref{realpair}), which is a means of transferring our knowledge of $E(n)$-cohomology over to $ER(n)$-cohomology.
\end{remark}

\medskip
\begin{remark}
Since $ER(2)^*(\Ek(\Z, m+1))$ and $ER(2)^*(\Ek(\Z/2^q, m))$ are trivial if $m \geq 2$, the present paper together with the computation of $ER(2)^*(\mathbb{C}P^\infty)$ in \cite{Lor16} and the computation of $ER(2)^*(\Ek(\Z/2^q,1))$ in \cite{KLW16a} 
completely describes the $ER(2)$-cohomology of Eilenberg-MacLane spaces.
\end{remark}

\medskip
\begin{remark}
As shown in \cite{HM}, the theory $ER(2)$ is (additively) equivalent to a version of topological modular forms with level structure, $TMF_0(3)$, after 
suitable completion.
\end{remark}

\medskip 
\noindent
We also relate our computations by means of some short exact sequences.

\medskip
\begin{defn}Given maps of graded, augmented, topological $ER(n)^*$-algebras \\* $A \lra B \lra C$, we say that the sequence $1 \lra A \lra B \lra C \lra 1$ is a \emph{short exact sequence of completed graded $ER(n)^*$-algebras} if the following is a short exact sequence of $ER(n)^*$-modules:
\[ 0 \lra B \hat{\otimes} I(A) \longrightarrow B \lra C \lra 0, \]
where $I(A)$ denotes the augmentation ideal of $A$, and the completed tensor product is taken over $ER(n)^*$. 
\end{defn}

The values of $ER(n)$-cohomology on the above spaces are compatible in the following sense:

\begin{thm} \label{main4}
There exist short exact sequences of completed $ER(n)^*$-algebras:
$$ 1 \lra ER(n)^*(\Ek(\Z/2,1)) \lra ER(n)^*(\BOg) \lra ER(n)^*(\BSOg) \lra 1, $$
$$ 1 \lra ER(n)^*(\Ek(\Z/2,2)) \lra ER(n)^*(\BSOg) \lra ER(n)^*(\BSg) \lra 1, $$
$$ 1 \lra ER(n)^*(\BSg) \lra ER(n)^*(\widetilde{\BSg}) \lra ER(n)^*(\Ek(\Z,3)) \lra 1, $$
$$ 1 \lra ER(n)^*(\Ek(\Z/2,3)) \lra ER(n)^*(\widetilde{\BSg}) \lra ER(n)^*(\BOg\langle 8 \rangle) \lra 1. $$ 
\end{thm}
\smallskip 

\begin{remark} \label{Pontrjagin}
Note in particular that for $n \leq 2$, the spaces $\BOg \langle 8 \rangle$ and $\widetilde{\BSg}$ are $ER(n)$-equivalent. This follows from the fact that $\Ek(\Z/2, 3)$ is $E(n)$ and hence $ER(n)$-acyclic for $n \leq 2$. 
\end{remark}

\begin{remark} 
Laures and Olbermann \cite{LO} have essentially established the above result for the $TMF_0(3)$-cohomology of $Z = \BOg\langle 8 \rangle$ using techniques of \cite{KW2}. However, their result only holds after $K(2)$-localization.
\end{remark}

\section{Preliminaries and the general result}

\medskip
We begin by recalling some definitions and notation. We will work with both $\Z/2$-spaces and $\Z/2$-spectra. Unless noted otherwise, we will use the letter $Y$ to denote a $\Z/2$-space. $Z$ will also denote a space, but usually with trivial $\Z/2$-action. When we refer to fixed points or homotopy fixed points, we will always mean with respect to the group $\Z/2$.

\medskip
\noindent
A (genuine) $\Z/2$-equivariant spectrum $\E$ is a family of $\Z/2$-spaces $\underline{\E}_{a+b\alpha}$, indexed over elements $a+b\alpha$ of the real representation ring of $\Z/2$, $RO(\Z/2)$, where $\alpha$ denotes the sign representation, together with a compatible system of equivariant homeomorphisms
$$\xymatrix{ \underline{\E}_{a-r+(b-s)\alpha} \ar@{->}[r]^-{\simeq} & \Omega^{r+s\alpha} \underline{\E}_{a+b\alpha}}$$
where the right hand side denotes the space of pointed maps (endowed with the conjugation action) from the one point compactification of the representation $r+s\alpha$. The reader may refer to \cite{HK} for more details on $\Z/2$-equivariant spectra. We will denote by $ER$ the homotopy fixed
point spectrum of the $\Z/2$-action on $\E$. A canonical example of a $\Z/2$-spectrum is  Real cobordism, $\MU$, whose underlying nonequivariant spectrum is complex cobordism, $MU$, studied first by Landweber \cite{Lan68} and Araki and Murayama \cite{AM78, Ara79a, Ara79b}, and more recently by Hu-Kriz \cite{HK}. The action of $\Z/2$ is induced by the complex conjugation action on the pre-spectrum representing $\MU$ in the usual way. 

\medskip
\noindent
The $p=2$ Johnson-Wilson theory $E(n)$
lifts to a $\Z/2$-spectrum, $\E(n)$, defined as an $\MU$ module by coning off certain equivariant lifts of the Araki generators $v_i$ for $i > n$, and then inverting the lift of $v_n$. We shall call these equivariant lifts by the same names, $v_i$. The Real Johnson-Wilson theories, $ER(n)$, are defined as the homotopy fixed points of $\E(n)$. In addition to $\E(n)$, we will also make use of the $\Z/2$-spectra $\BP\langle n \rangle$, whose underlying nonequivariant spectra are the truncated Brown-Peterson theories $BP\langle n \rangle$ (with coefficients $BP\langle n \rangle_*=\Z_{(2)}[v_1, \dots, v_n]$). In particular, we will be interested in the spaces which comprise these spectra, $\BPRm{n}{s+t\alpha}
$ for various $n, s,$ and $t$.

\medskip
\noindent
The bi-degrees of the equivariant lifts are given by $|v_i| = (2^i-1)(1+\alpha)$. By diagonal classes, we will mean homotopy classes in bi-degrees of the form $k(1+\alpha)$. In \cite[Claim 4.1]{Nitufib}, the first and third authors observed an invertible class $y(n) \in \pi_{\lambda + \alpha} \, \E(n)$, with 
$\lambda=\lambda(n) = 2^{n+2}(2^{n-1}-1)+1$. 
We will use this class to shift diagonal classes to integral degrees as we will now describe.

\medskip
\noindent
Working with cohomological grading, let $\E$ be a $\Z/2$-equivariant ring spectrum with underlying spectrum $E$. Assume that $Y$ is a $\Z/2$-space and let $\E^{\ast(1+\alpha)}(Y)$ denote the sub-ring of diagonal elements in the equivariant $\E$-cohomology of $Y$ i.e. $\E^{\ast(1+\alpha)}(Y) := \pi_0\Maps^{\Z/2}(Y,\underline{\E}_{\ast(1+\alpha)})$.
\medskip
\noindent

\begin{remark}\label{remark:mult}
In \cite[Theorem 1.2]{KLW16b} we show that, restricted to the category of equivariant spaces, $\E(n)$ naturally gives rise to a multiplicative cohomology theory valued in commutative $\MU$-algebras. We further show that the map given by forgetting the $\Z/2$-action 
$$\rho: \E^{\ast(1+\alpha)}(n)(-) \longrightarrow E^{\ast}(n)(-)$$
is an algebra homomorphism over $\MU^{\ast(1+\alpha)}(-)$. Note that a stronger statement appears in \cite[Comment 5 on page 349]{HK} where the authors claim that $\E(n)$ is a homotopy commutative, homotopy associative ring spectrum. Even though we have not been able to verify that statement, the weaker multiplicative structure described in \cite{KLW16b} suffices to prove all results in this document as well as those in previous articles by the authors.
\end{remark}

Now, given a class $z \in \mathbb{E}(n)^{k(1+\alpha)}(Y)$, we may define the product $\widehat{z}:=zy(n)^{k}$, which lives in integral degree $k(1-\lambda)$. Since the underlying nonequivariant class of $y(n)$ is $v_n^{2^n-1}$, this map is an equivariant lift of the endomorphism given in Definition \ref{defhat} in the following sense. Letting $\rho : \E^{\ast(1+\alpha)}(Y) \longrightarrow E^{\ast}(Y)$ denote the forgetful ring homomorphism as in the remark above, we notice that the image of $\rho$ belongs to the graded sub-group of elements in even degree. We now have the commutative diagram
$$\xymatrix{\E(n)^{*(1+\alpha)}(Y) \ar@{->}[r]^-{\widehat{ }} \ar@{->}[d]^-{\rho} & \E(n)^{*(1-\lambda)}(Y) \ar@{->}[d]^-{\rho} & \text{if }|z|=k(1+\alpha),\text{ }z\mapsto \widehat{z}:=zy(n)^k\\
E(n)^{2*}(Y) \ar@{->}[r]^-{\widehat{ }} & E(n)^{*(1-\lambda)}(Y) &  \text{if }|z|=2k,\text{ }z\mapsto \widehat{z}:=zv_n^{k(2^n-1)}}$$}

\noindent
\begin{defn} \label{completion2}
 For a space or spectrum $Y$, we call $\Maps(E\Z/2_+, Y)$ the completion of $Y$ and say that $Y$ is complete if the map $Y \longrightarrow \Maps(E\Z/2_+, Y)$ is a $\Z/2$-equivalence. 
\end{defn}

\medskip
\noindent
As shown by Hu and Kriz in \cite[Theorem 4.1]{HK}, 
the spectra $\E(n)$ are complete. Thus, the homotopy 
fixed point spectra $ER(n) := \E(n)^{h\Z/2}$ agree with the genuine fixed point spectra $\E(n)^{\Z/2}$. By definition we have: $\pi_\ast ER(n) = \pi_{*+0\alpha} \E(n)$. It is shown in \cite[Section 4]{Nitufib} that the homotopy ring of $ER(n)$ is a subquotient of a ring:
\[ \Z_{(2)}[x, \hat{v}_1, \hat{v}_2, \ldots, \hat{v}_{n-1},v_n^{\pm 1}], \] 
where $x$ is an element of $\pi_\lambda ER(n)$. It is $2$-torsion and $x^{2^{n+1}-1}=0$. The classes $\hat{v}_k \in ER(n)_{2^{n+2}(1-2^{n-1})(2^k-1)}$ for $k \leq n$ are the hatted representatives of the diagonal Araki classes mentioned above. The class $\hat{v}_k$ maps to the class  $v_k v_n^{-(2^n-1)(2^k-1)}$ 
under the canonical map from $ER(n)_\ast$ to $E(n)_\ast$. 

\medskip
\noindent
The following property will be essential to generating permanent cycles in our computations.

\smallskip
\begin{defn} A $\Z/2$-space $Y$ is said to have the weak projective property with respect to a $\Z/2$-spectrum $\E$ if the forgetful map $\rho$:
\[ \rho : \E^{\ast(1+\alpha)}(Y) \longrightarrow E^{2\ast}(Y), \]
is an isomorphism of graded abelian groups. 
\end{defn}

\smallskip

\noindent
Let us relate spaces with the weak projective property to the definitions from \cite{NituER2}:

\medskip
\begin{defn} A $\Z/2$-space $X$ is said to be projective if 
\begin{enumerate}
\item $H_*(X; \Z)$ is of finite type. 
\item $X$ is homeomorphic to $\bigWedge_I ({\C P^\infty})^{\Smash k_I}$ 
for some weakly increasing sequence of integers $k_I$, 
with the $\Z/2$ action given by complex conjugation. 
\end{enumerate}
\end{defn}

\smallskip 

By a $\Z/2$-equivariant $H$-space, we shall mean an $H$-space whose multiplication map is $\Z/2$-equivariant.
\medskip 

\begin{defn} A $\Z/2$-equivariant $H$-space $Y$ is said to have the projective 
property if there exists a projective space $X$, along 
with a pointed $\Z/2$-equivariant map $f : X \longrightarrow Y$, 
such that $H_*(Y;\Z/2)$ is generated as an algebra by the image of $H_*f$.
\end{defn} 

\noindent
Spaces with the projective property are not rare. Examples include $\underline{\MU}_{k(1+\alpha)}$,
$\underline{\B\D}_{k(1+\alpha)}$, and
$\underline{\B\D}\langle n \rangle_{k(1+\alpha)}$ where this last is
only for $k < 2^{n+1}$ (see \cite[Theorem 1.3]{RW2}, \cite[Theorem 2.2]{HH}, and the remarks following Definition 1.4 in \cite{NituER2}).  These spaces are known to be complete (see \cite[Theorem 4.1]{HK} and the appendix in \cite{Kitch-Wil-split}). 

\smallskip

\noindent
As suggested by our notation, spaces with the projective property also have the weak projective property:

\medskip
\begin{thm} \label{main1}
Let $Y$ be a $\Z/2$-equivariant $H$-space with the projective property. Let $\E$ denote any complete $\MU$-module spectrum with underlying spectrum $E$, satisfying the property that the forgetful map: $\rho^\ast : \E^{\ast(1+\alpha)} \longrightarrow E^{2\ast}$, is an isomorphism. Then the space $Y$ has the weak projective property with respect to $\E$. In other words, the following map is an isomorphism of $\MU^{\ast(1+\alpha)}$ modules: 
\[ \rho : \E^{\ast(1+\alpha)}(Y) \longrightarrow E^{2\ast}(Y). \]
\end{thm}

\smallskip

\begin{remark}\label{completion}
Notice that if $Y$ is complete, then its homotopy fixed points, $Y^{h\Z/2}$, are equivalent to its fixed points, $Y^{\Z/2}$. Consequently, for a complete space $Y$, the (equivariant) inclusion $Y^{\Z/2} \subseteq Y$ may be replaced by the inclusion of the homotopy fixed points $Y^{h\Z/2}$ (endowed with the trivial action) into the  space $Y$. This offers us a technical advantage since it allows us to invoke the homotopy fixed point spectral sequence (see Section \ref{sec:examples}). We will therefore implicitly assume for the the rest of the paper that all spaces $Y$ with the projective property are already complete. Indeed, any possibly uncompleted spaces $Y$ with the projective property can be completed without changing $\E^{*(1+\alpha)}(Y)$ or $E^{2\ast}(Y)$. So one may complete a space with the projective property (unless already complete) in order to ensure the technical requirement above.
\end{remark}

\noindent
Two complete spectra that satisfy the conditions of Theorem \ref{main1}, $\Maps(E\Z/2_+, \B\D\langle n \rangle)$ and $\E(n)$, will be relevant to us in this paper. 

\medskip
\noindent
Now given a space $Y$ with the weak projective property, we may invert $\rho$ and then `hat' elements to construct an isomorphism of rings that scales degrees by $(1-\lambda)/2$:
\[ \psi_Y : E(n)^{2\ast}(Y)  \longrightarrow \E(n)^{\ast(1-\lambda)}(Y), \quad \quad z \longmapsto \widehat{\rho^{-1}(z)} \]
Restricting to $Y=pt$, we observe that $\psi(v_k):=\psi_{pt}(v_k)=\hat{v}_k$ for $k \le n$. 
Let $\hat{E}(n)^* \subseteq ER(n)^*=\mathbb{E}(n)^*(pt)$ denote the image of $\psi$. Then $\psi$ gives an abstract isomorphism of rings (but not graded rings) between $E(n)^*$ and $\hat{E}(n)^*$.

\medskip
\begin{remark}
As in Remark \ref{anomaly}, we have to treat the case $n=1$ separately. In this special case, we define the domain of $\psi_Y$ to be $E(1)^0(Y)$ in order to ensure that the map is injective and the following theorems still hold. 
\end{remark} 

\medskip
\noindent
We now get to the main definition of this article:

\medskip
\begin{defn} \label{realpair}
Let $Z$ be a (nonequivariant) space. Suppose there exists a $\Z/2$-equivariant space $Y$ with the weak projective property with respect to $\E(n)$, equipped with a map $g : Z \longrightarrow Y^{\Z/2}$ which satisfies the following properties:
\begin{enumerate}[(i)]
\item The composite $\xymatrix{Z \ar@{->}[r]^-g & Y^{\Z/2} \ar@{->}[r]^-i & Y}$, where $i$ denotes the inclusion of fixed points, induces a surjection on $E(n)^*(-)$.
\item The composite $\chi_Z$, defined as the restriction of $\psi_Y$ above along $(i \circ g)^*$,
$$\xymatrix{E(n)^{2\ast}(Y) \ar@{->}[r]^-{\psi_Y} & \E(n)^{\ast(1-\lambda)}(Y) \ar@{->}[r]^-{(i \circ g)^*} & \E(n)^{\ast(1-\lambda)}(Z) \ar@{=}[r] & ER(n)^{\ast(1-\lambda)}(Z)}$$
factors through the map $\xymatrix{E(n)^{2\ast}(Y) \ar@{->}[r]^-{(i \circ g)^*} & E(n)^{2 \ast}(Z)}$. 

\noindent
Recall that for $n=1$, by definition, we restrict the domain of $\psi_Y$ to $E(1)^0(Y)$. 

\end{enumerate}
Then we call the pair $(Y,Z)$ a \emph{real pair with respect to $E(n)$}. If in addition, the $2$-completed $BP$ cohomology of $Z$ is Landweber flat, then we say that the real pair is \emph{Landweber flat}.
\end{defn}



\medskip
\noindent
The main result of this article is now straightforward to state.

\medskip
\begin{thm} \label{main2}
Let $(Y,Z)$ be a Landweber flat real pair with respect to $E(n)$. Let $\hat{E}(n)^*(Z)$ denote the image of $\chi_Z$ (above) and let $\varphi_Z: E(n)^{2\ast}(Z) \longrightarrow ER(n)^{\ast(1-\lambda)}(Z)$ denote the factorization. Then $\varphi_Z$ is injective and induces an isomorphism of algebras:
\[ ER(n)^* \otimes_{\hat{E}(n)^*} \hat{E}(n)^*(Z) \longrightarrow ER(n)^*(Z). \]
Note that the notation for the subalgebra $\hat{E}(n)^*(Z)$ of $ER(n)^*(Z)$ does not conflict with the subalgebra of $E(n)^*(Z)$ of the same name in Definition \ref{defhat} as the two subalgebras are isomorphic via $ER(n)^*(Z) \longrightarrow E(n)^*(Z)$.
\end{thm}

\medskip

\begin{remark} Notice that as a subalgebra of $ER(n)^*(Z)$, $\hat{E}(n)^*(Z)$ does
depend on the space $Y$. However, as noted above, it is abstractly isomorphic to $E(n)^*(Z)$ after a suitable rescaling of the degrees, and its image as a subalgebra of $E(n)^*(Z)$ does not depend on $Y$. See also Remark \ref{chooseY}.
\end{remark}

\begin{remark}
In the statement of the above theorem we follow the convention introduced in \cite[Definition 1.5]{RWY} of replacing $E(n)$ by its 2-adic completion when $\lim^1 BP(Z^m)$ is non-zero, where $Z^m$ denote the finite skeleta of $Z$. This allows us to eradicate any phantom maps in cohomology. All arguments presented in this article also hold after 2-completion and so we may conveniently suppress the completion from the notation. Additionally, as in \cite{RWY}, for any cohomology theory $E$, we denote by $E^*(Z_1) \hat{\otimes} E^*(Z_2)$ the inverse limit of the tensor products (over $E_*$) of the $E$-cohomology of the (skeletal) filtration quotients of $E^*(Z_1)$ and $E^*(Z_2)$.
\end{remark}

\medskip
\noindent
With the above convention understood, we will also prove the following useful theorem which allows us to generate more computational examples by taking smash products. This lifts the $E(n)$-based Kunneth isomorphism for spaces whose $BP$-cohomology is Landweber flat (which follows from \cite[Theorem 1.8 and Proposition 1.9]{W99}).

\medskip
\begin{thm} \label{Kunneth}
Assume $(Y_a,Z_a)$, and $(Y_b, Z_b)$ are two Landweber flat real pairs with respect to $E(n)$. Then the pair: $(Y_a \wedge Y_b, Z_a \wedge Z_b)$ is a Landweber flat real pair, and the completed Kunneth isomorphism (over the coefficients) holds in (reduced) $ER(n)$-cohomology: 
\[ \tilde{ER}(n)^*(Z_a \wedge Z_b) = \tilde{ER}(n)^*(Z_a) \widehat{\otimes} \tilde{ER}(n)^*(Z_b). \]

\end{thm}

\medskip
\noindent
In this article, we will show that several familiar spaces belong to Landweber flat real pairs. This, together with Theorem \ref{main1}, will give the content of Theorem \ref{main3} above. As examples, we have that, for arbitrary natural numbers $m$ and $q$, the spaces $Z = \Ek(\Z, 2m+1)$, $\Ek(\Z/2^q, 2m)$ and $\EMf$ are parts of Landweber flat real pairs with respect to $E(n)$ for arbitrary $n$. In addition, the connective covers $\BOg, \BSOg, \BSg, \widetilde{\BSg},$ and $\BOg\langle 8 \rangle$ also belong to Landweber flat real pairs with respect to $E(n)$, with the last case only holding for $n \leq 2$. Here the space $\widetilde{\BSg}$ denotes the fiber of the map $p_1: \BSg \longrightarrow \Ek(\Z, 4)$ and it belongs to the fibrations
$$\Ek(\Z, 3) \longrightarrow \widetilde{\BSg} \longrightarrow \BSg \text{\hspace{5pt} and \hspace{5pt}} \BOg\langle 8 \rangle \longrightarrow \widetilde{\BSg} \longrightarrow \Ek(\Z/2, 3).$$
\medskip

\section{Background: The Bockstein spectral sequence}

\medskip
\noindent
The ideas used in this article were implicit in the paper \cite{KW2} where the first and third authors computed $ER(n)^*(\BOg(q))$. The main tool that was exploited was the internal structure of a spectral sequence known as the Bockstein spectral sequence. Let us briefly recall the construction of this spectral sequence: 

\noindent
In \cite[Theorem 1.6]{Nitufib}, the first and third authors constructed a fibration of spectra:
\[
\Sigma^\lambda ER(n) \lra ER(n) \lra E(n)
\]
with the first map given by multiplication with the $(2^{n+1}-1)$-nilpotent class $x$ described above. Consequently this fibration leads to a convergent 
spectral sequence which we call the Bockstein spectral sequence. The main properties of this spectral sequence are described explicitly by the following theorem. For further properties of this spectral sequence, see \cite[Theorem 2.1]{KW2}.

\medskip
\begin{thm} \cite{KW2}
\label{bss} 
For $X$ a spectrum,
the above fibration yields a first and fourth quadrant 
spectral sequence of $ER(n)^\ast$-modules, $\mbox{E}_r^{i, j}(X) 
\Rightarrow ER(n)^{j-i}(X)$. The differential $d_r$ has 
bi-degree $(r, r+1)$ for $r \geq 1$.
The $\mbox{E}_1$-term is given by: $\mbox{E}_1^{i,j}(X) = 
E(n)^{i\lambda +j-i}(X)$, with 
\[ d_1(z) = v_n^{-(2^n-1)}(1-c)(z), \,  
\text{ where }  \, 
c(v_i)= -v_i.
\]
The differential $d_r$ 
increases cohomological degree by $1+r\lambda$
between the appropriate 
sub-quotients of $E(n)^\ast(X)$. 

\end{thm}
\begin{remark}\label{E1}
When $X$ is a space,
notice that there is a canonical 
class in $\mbox{E}_1^{1,-\lambda +1}(X)$ corresponding to the unit element under the identification 
$\mbox{E}_1^{1,-\lambda+1}(X) = E(n)^0(X)$. 
This class represents the element $x \in ER(n)^{-\lambda}$, and we call it by the same name.
It is a permanent cycle and we may use this 
class to simplify the notation. 
$$
\mbox{E}_1^{\ast, \ast}(X) = \mbox{E}_1^{0,\ast}(X)[x] = E(n)^\ast(X)[x], 
\quad |x|=(1,-\lambda +1). 
$$
$$
d_1(z) = x \, v_n^{-(2^n-1)}(1-c)(z), \,  
\text{ where }  \, 
v_n \in \mbox{E}_1^{0,-2(2^{n}-1)}.
$$

\end{remark}

\begin{remark}Depending on whether one truncates the multiplication by $x$ tower, there are two spectral sequences that can arise from the above. One converges to $ER(n)^*(X)$ (as in \cite{KW2, Lor16}), the other to $0$ (as in \cite{NituP, NituP2, KLW16a}). In the latter case, one must go back to reconstruct the answer from the differentials. Both have their advantages and ultimately contain equivalent information, but it is the truncated BSS converging to $ER(n)^*(X)$ that we use in this paper.
\end{remark}

Let us now recall the structure on this spectral sequence as developed in \cite{KW2}. The Bockstein spectral sequence is a spectral sequence of modules over
$ER(n)^*$.  Recall the map $\psi : E(n)^* \longrightarrow ER(n)^*$ above defined by $\psi(v_k) = \hat{v}_k$ for $k \le n$, with $\hat{E}(n)^*$ being defined as the image of $\psi$. Recall also that $\hat{E}(n)^*$ is abstractly isomorphic to $E(n)^*$ via the map $\psi$ that scales degrees by $(1-\lambda)/2$. It was shown in \cite{KW2} that the entire Bockstein spectral sequence for a point belongs to the category of $E(n)_*E(n)$-comodules which are finitely presented as $E(n)^*$-modules, which implies the following theorem:

\medskip
\begin{thm} \cite[Theorem 4.3]{KW2}
\label{ss}
Suppose $M$ is Landweber flat $\hat{E}(n)^\ast$-module, 
and let $(\mbox{E}_\ast,d_\ast)$ denote the 
Bockstein spectral sequence for $X = pt$. 
Then $(M \otimes_{\hat{E}(n)^*} \mbox{E}_*,\id_M 
\otimes_{\hat{E}(n)^*} d_\ast)$ is a spectral sequence 
of $ER(n)^\ast$-modules that 
converges to $M \otimes_{\hat{E}(n)^*} ER(n)^*$.
\end{thm}

\section{Proof of the main theorems}

\noindent
We now proceed to prove the main theorems stated in the introduction. 

\begin{proof} (of Theorem \ref{main1}). Recall that by definition of the projective property, there is a projective space $X$ so that there is a $\Z/2$-equivariant pointed map $f: X \longrightarrow Y$, whose image generates the mod 2 homology. It follows that $H_*(Y, \Z/2)$ is even, and so $H_*(Y,\Z_{(2)})$ is free. The James construction produces a $\Z/2$-equivariant space $JX$ and, since $Y$ is an equivariant $H$-space, the map $f:X \longrightarrow Y$ extends to a map $\lift{f}: JX \longrightarrow Y$ given by sending an $m$-tuple to the product of its components in $Y$. The map $\lift{f}$ is evidently $\Z/2$-equivariant and induces a surjection on $H_*(-; \Z_{(2)})$. The Atiyah-Hirzebruch spectral sequence now shows that $MU_\ast (JX)$, and $MU_\ast (Y)$ are free $MU_\ast$-modules, and the map $MU_\ast(JX) \longrightarrow MU_\ast(Y)$ is split surjective. We may choose a basis $\{\gamma_1, \gamma_2, \dots\}$ for a subspace of $MU_\ast(JX)$ on which $MU_\ast(\lift{f})$ is an isomorphism by choosing it on the $E_2$-page of the spectral sequence. The next step is to pick equivariant representatives for these classes. 

The stable splitting $\Sigma^\infty JX \simeq \bigvee_{j \geq 1} (\Sigma^\infty X)^{\Smash j}$ is $\Z/2$-equivariant (see \cite[VII.5]{LMS}). Now consider the spectral sequence constructed using the cellular filtration of $JX$ induced by the canonical (equivariant) cellular filtration of the projective space $X$, and converging to $\MU_\star(JX)$ (i.e. the $\Z/2$-equivariant Atiyah-Hirzebruch spectral sequence of \cite{HK}). In the above stable splitting, each summand is projective, and since $\MU$ is Real-oriented, the spectral sequence collapses for projective spaces. This yields that $\MU_\star(JX)$ is a free $\MU_\star$-module on diagonal generators. Forgetting the $\Z/2$-action, these become the generators of $MU_\ast(JX)$, so we may choose a subset of the diagonal generators which lift the $\gamma_i \in MU_\ast(JX)$ above. We call these by the same name.

Since $\MU_\star(JX)$ is free, $\MU \wedge JX$ is a free $\MU$-module spectrum on a generating set of finite type $\MU \wedge JX = \bigvee_i \Sigma^{k_i(1+\alpha)} \MU$ where $\{ k_1, k_2, \ldots \}$ is a weakly increasing sequence of non-negative integers. The degrees of the (diagonal) $\gamma_i$ give a suitable subsequence $\{ \beta_1, \beta_2, \ldots \}$ in $\{ k_1, k_2, \ldots \}$ so that we get a $\Z/2$-equivariant map:
\[ \bigvee \Sigma^{\beta_i(1+\alpha)} \MU \longrightarrow Y \wedge \MU \]
which is a (non-equivariant) equivalence. On freeing up our spectra, it follows that we have an equivariant equivalence of $\MU$-module spectra
\[ \bigvee \Sigma^{\beta_i(1+\alpha)} E\Z/2_+ \wedge \MU \longrightarrow E\Z/2_+ \wedge Y \wedge \MU. \]
Mapping out of the above equivalence in the category of $\MU$-module spectra, and using the completeness of $\E$ we have that $\E^{\star}(Y)$ is a free module over $\E^{\star}$ on the classes $\gamma_i$:
\[ \E^\star \, (Y) =  \E^\star \langle \langle \gamma_1, \gamma_2, \ldots \rangle \rangle .\]
It follows that the forgetful map $\rho$ maps $\E^{\ast(1+\alpha)}(Y)$ isomorphically onto $E^{2\ast}(Y)$. 
 \end{proof}

\begin{remark}\label{free}
Though beyond the conclusion of Theorem \ref{main1} we do not use this fact here, we note that the above proof in fact computes the entire $RO(\Z/2)$-graded $\E$-cohomology of $Y$.
\end{remark}

\begin{remark}\label{rmk:smash}
Note that if $Y_a$ and $Y_b$ are spaces with the projective property, then $Y_a \wedge Y_b$ is not necessarily an $H$-space and thus may not have the projective property. However, the above proof still shows that $Y_a \wedge Y_b$ does have the weak projective property. This is used in the proof of Theorem \ref{Kunneth} below.
\end{remark}

\noindent
Let us proceed to the proof of Theorem \ref{main2}. 

\begin{proof} (of Theorem \ref{main2}). We will first prove the theorem for $n>1$, then discuss the changes necessary to prove it for $n=1$. Let $Z$ be a space whose 2-completed $BP$-cohomology is Landweber flat. Then, by \cite[Proposition 2.1]{W99}, $E(n)^*(Z)$ is Landweber flat with respect to $E(n)$ as well. Assume that $Z$ is endowed with a map $g : Z \longrightarrow Y^{\Z/2}$ for some space $Y$ with the weak projective property and so that $E(n)^*(Y)$ surjects onto $E(n)^*(Z)$. Finally, assume that the map:
\[ \chi_Z : E(n)^{2\ast}(Y) \longrightarrow \E(n)^{\ast(1-\lambda)}(Y) \longrightarrow ER(n)^{\ast(1-\lambda)}(Z), \]
factors through $E(n)^{2\ast}(Z)$ and let $\varphi_Z : E(n)^{2\ast}(Z) \longrightarrow ER(n)^{\ast(1-\lambda)}(Z)$ denote this factorization:

$$\xymatrix{E(n)^{2\ast}(Y) \ar@{->}[r]^-{\chi_Z} \ar@{->}[d]_-{(i \circ g)^*} & ER(n)^{\ast (1-\lambda)}(Z) \\
E(n)^{2 \ast}(Z) \ar@{->}[ur]_-{\varphi_Z} & }$$

 The composite of $\varphi_Z$ and the canonical map $\iota : ER(n)^{\ast(1-\lambda)}(Z) \longrightarrow E(n)^{\ast(1-\lambda)}(Z)$ is nothing other than the algebra homomorphism
\[ \iota \varphi_Z : E(n)^{2\ast}(Z) \longrightarrow E(n)^{\ast(1-\lambda)}(Z), \quad \quad z \longmapsto zv_n^{(2^n-1)|z|/2}. \]
which is injective since $v_n$ is invertible (and $1-\lambda \neq 0$). From this it follows that $\varphi_Z$ is injective. Now notice by definition that the sub algebra $\hat{E}(n)^*(Z) \subseteq ER(n)^*(Z)$ consists of permanent cycles. In addition, $\hat{E}(n)^*(Z)$ is Landweber flat. Hence, invoking Theorem \ref{ss}, we get a map of spectral sequences:
\[ \mbox{E}_r  \otimes_{\hat{E}(n)^*} \hat{E}(n)^*(Z) \longrightarrow \mbox{E}_r(Z), \]
representing the canonical map $ER(n)^* \otimes_{\hat{E}(n)^*} \hat{E}(n)^*(Z) \longrightarrow ER(n)^*(Z)$, 
and where $\mbox{E}_r$ denotes the Bockstein spectral sequence for a point, and $\mbox{E}_r(Z)$ denotes the Bockstein spectral sequence for the space $Z$. We claim that this map is an isomorphism at the $\mbox{E}_1$-stage, and thus at the $\mbox{E}_\infty$-stage. Indeed, using the description of the $\mbox{E}_1$-page in Remark \ref{E1} and the isomorphism $E(n)^*(Z) \cong E(n)^* \otimes_{\hat{E}(n)^*} \hat{E}(n)^*(Z)$, we have
\begin{align*}
E_1(Z) & \cong E(n)^*(Z)[x]\\
& \cong (E(n)^* \otimes_{\hat{E}(n)^*} \hat{E}(n)^*(Z))[x] \\
& \cong E(n)^*[x] \otimes_{\hat{E}(n)^*} \hat{E}(n)^*(Z)\\
& \cong \mbox{E}_1  \otimes_{\hat{E}(n)^*} \hat{E}(n)^*(Z)
\end{align*}
Thus, this map of spectral sequences represents an isomorphism at the $\mbox{E}_\infty$-stage, and hence must represent an isomorphism.

If $n=1$ (and hence $1-\lambda=0$), then the following changes are necessary. The domain of $\psi_Y$ is defined to be $E(1)^0(Y)$, and we assume that the map 
$$\chi_Z: E(1)^0(Y) \longrightarrow ER(1)^0(Z)$$
 factors through $E(1)^{0}(Z)$. We let $\varphi_Z: E(1)^0(Z) \longrightarrow ER(1)^0(Z)$ denote this factorization. The map 
 $$\iota \varphi_Z: E(1)^0(Z) \longrightarrow E(1)^0(Z)$$
  is then the identity map when $n=1$ and the subalgebra $\hat{E}(n)^*(Z)$ of $E(n)^*(Z)$ is replaced by $E(1)^0(Z)$ (and $\hat{E}(n)^*$ is replaced by $E(1)^0$). As $E(1)^0(Z)$ is Landweber flat, we obtain a map of spectral sequences 
  $$\mbox{E}_r \otimes_{E(1)^0} E(1)^0(Z) \longrightarrow \mbox{E}_r(Z).$$
   The same argument as above (in particular, the fact that $E(1)^*(Z) \cong E(1)^* \otimes_{E(1)^0} E(1)^0(Z)$) then gives the desired isomorphism.

\end{proof}
\begin{proof} (of Theorem \ref{Kunneth}). First, observe that the space $Y_a \wedge Y_b$ has the weak projective property by Remark \ref{rmk:smash}. Also, the Kunneth isomorphism in Morava K-theory implies that $Z_a \wedge Z_b$ has even Morava $K$-theory. It follows by \cite[Lemma 5.5]{RWY} that the $2$-completed $BP$ cohomology of $Z_a \wedge Z_b$ is Landweber flat. Note that by \cite[Proposition 2.1]{W99}, it follows that the $E(n)$-cohomology of $Z_a \wedge Z_b$ is Landweber flat with respect to $E(n)$ as well.

\smallskip
\noindent
Define the map $g$ by the composite
$$\xymatrix{Z_a \wedge Z_b \ar@{->}[r]^-{g_a \wedge g_b} & Y_a^{\Z/2} \wedge Y_b^{\Z/2} \ar@{->}[r]&  (Y_a \wedge Y_b)^{\Z/2}}$$
We will now prove that (i) and (ii) of Definition \ref{realpair} hold. Consider the cofiber sequences
$$\xymatrix{Z_a \ar@{->}[r]^-{i \circ g_a} & Y_a \ar@{->}[r]^-{h_a} &C(i \circ g_a)} \text{ and } \xymatrix{Z_b \ar@{->}[r]^-{i \circ g_b} & Y_b \ar@{->}[r]^-{h_b} &C(i \circ g_b)}$$
By our assumptions, both of these give rise to short exact sequences in $E(n)$-cohomology. It now follows from the fact that all spaces under consideration have evenly graded $E(n)$-cohomology that the above cofibrations give rise to short exact sequences of evenly graded modules in $K(n)$-cohomology. Smashing with the identity maps on $Y_b$ and $Z_a$, respectively, the following give rise to short exact sequences in $K(n)$-cohomology:
$$\xymatrix{Z_a \wedge Y_b \ar@{->}[r]^-{(i \circ g_a) \wedge 1} & Y_a \wedge Y_b \ar@{->}[r]^-{h_a\wedge 1} &C(i \circ g_a)} \wedge Y_b \quad \text{ and } \quad \xymatrix{Z_a \wedge Z_b \ar@{->}[r]^-{1 \wedge (i \circ g_b)} & Z_a \wedge Y_b \ar@{->}[r]^-{1 \wedge h_b} &Z_a \wedge C(i \circ g_b)}$$
Invoking \cite[Lemma 7.1]{RWY}, we see that these also induce short exact sequence in (reduced) $E(n)$-cohomology. This establishes the following natural isomorphisms:
\[ \tilde{E}(n)^*(Z_a \wedge Z_b) = \frac{{\tilde{E}(n)}^*(Y_a \wedge Y_b)}{\langle h_a^* + h_b^* \rangle} = \frac{{\tilde{E}(n)}^*(Y_a) \widehat{\otimes} \tilde{E}(n)^*(Y_b)}{\langle h_a^* + h_b^*\rangle}, \]
with the second isomorphism following from the Atiyah-Hirzeburch spectral sequence for $Y_a$ and $Y_b$ which collapses due to the fact that they have evenly graded torsion-free singular cohomology. By naturality, $\psi_{Y_a \wedge Y_b} \circ (h_a \wedge 1)$ is null-homotopic, so $\psi_{Y_a \wedge Y_b}$ factors through $\tilde{E}(n)^*(Z_a \wedge Y_b)$. Call this factorization $f$. Again by naturality, we have that $f \circ (1 \wedge h_b)$ is null-homotopic, so $f$ factors through $\tilde{E}(n)^*(Z_a \wedge Z_b)$. This factorization is the desired $\varphi_{Z_a \wedge Z_b}$. This establishes the validity of conditions (i) and (ii) of Definition \ref{realpair}. It remains to show the Kunneth isomorphism for $ER(n)$-cohomology. Notice that for $E(n)$-cohomology, the Kunneth isomorphism is a straightforward consequence of Theorem 1.8 and Proposition 1.9 in \cite{W99}. The Kunneth isomorphism for $ER(n)$ then follows since $ER(n)^*$ is finitely generated over $\widehat{E}(n)^*$ (with an obvious adjustment in the case when $n=1$).
\end{proof}

\section{Examples}\label{sec:examples}

\medskip
We now apply our main theorem to some examples. As we will see, the spaces $\EM$, $\EMF$ and $\EMf$ , and the connective covers $\BOg, \BSOg, \BSg, \widetilde{\BSg}$ are all part of Landweber flat pairs with respect to $E(n)$. In each of our examples, we use an auxiliary space $Y_1$ with the projective property together with an equivariant $H$-map $h: Y \longrightarrow Y_1$ to produce the desired factorization $\varphi_Z$. Specifically, we will produce sequences of equivariant $H$-maps (where $Z$ has trivial $\Z/2$-action)
$$\xymatrix{Z \ar@{->}[r]^{i \circ g} & Y \ar@{->}[r]^{h} & Y_1}$$
satisfying the following properties:

\medskip
\begin{conditions}
\leavevmode
\begin{enumerate}[(1)]
\item $Y$ and $Y_1$ are complete spaces with the projective property,
\item The composite $h \circ (i \circ g)$ is equivariantly null-homotopic,
\item the sequence induces a left exact sequence of bicommutative Hopf algebras in Morava $K$-theory, $K(m)_*(-)$, for all $m >0$. 
\end{enumerate}
\end{conditions}

\medskip
\noindent
These conditions will ensure that $(Z, Y)$ is a real pair as follows. Conditions (1) and (3) show that $Z$ has even Morava $K(m)$-homology for $m>0$ (even Morava $K$-theory, for short). It follows that $E(n)^*(Z)$ is Landweber flat. We claim that the sequence of spaces $Z \longrightarrow Y \longrightarrow Y_1$ induces a right exact sequence of completed, graded $E(n)^*$-algebras. We will prove this from the following lemma, an application of \cite{RWY}.

\medskip
\begin{lemma} \label{K(n)toE(n)}
Suppose that the sequence of spaces
$$\xymatrix{X_1 \ar@{->}[r]^-{a} & X_2 \ar@{->}[r]^-{b} & X_3}$$
induces a right exact sequence (of $K(m)^*$-modules) in $K(m)^*$-cohomology for all $m>0$, $X_2$ and $X_3$ have even Morava $K$-theory, and $b \circ a \simeq \ast$. Then it induces a right exact sequence (of $E(n)^*(-)$-modules) in $E(n)$-cohomology. The same is true if right exact is replaced by short exact. 
\end{lemma}
\begin{proof}
Consider the cofiber sequence $\xymatrix{X_1 \ar@{->}[r]^-{a} & X_2 \ar@{->}[r]^-{c} & C(a)}$. Note that since $a^*$ is surjective in $K(m)^*(-)$ for all $m>0$, the above sequence must be short exact in $K(m)^*(-)$, and so $C(a)$ has even Morava $K$-theory and $c^*$ must be injective in $K(n)^*(-)$. It follows from \cite[Lemma 7.1]{RWY} that $c^*$ is injective in $E(n)^*(-)$. Hence, the sequence $X_1 \longrightarrow X_2 \longrightarrow C(a)$ is short exact in $E(n)^*(-)$. 

Since $b \circ a \simeq \ast$, we have a factorization of $b$ through $p: C(a) \longrightarrow X_3$. We claim that $p$ is surjective in $E(n)^*(-)$. A diagram chase shows it is surjective in $K(m)^*(-)$ for all $m>0$. Now consider the cofiber sequence $C(a) \longrightarrow X_3 \longrightarrow C(p)$. It must induce a short exact sequence in $K(m)^*(-)$ and so $C(p)$ has even Morava $K$-theory. In particular, in $K(n)$-cohomology we have a short exact sequence
$$\xymatrix{0 \ar@{->}[r] & K(n)^*(C(p)) \ar@{->}[r] & K(n)^*(X_3) \ar@{->}[r]^-{p^*} & K(n)^*(C(a)) \ar@{->}[r] & 0}$$
Since the left hand map is injective in $K(n)^*(-)$ and $C(p)$ and $X_3$ have even Morava $K$-theory, applying \cite[Lemma 7.1]{RWY} again shows that the left hand map is injective in $E(n)^*(-)$. It follows that the sequence $C(a) \longrightarrow X_3 \longrightarrow C(p)$ is short exact in $E(n)^*(-)$ and in particular $p^*$ is surjective. Splicing $X_3$ into the sequence $X_1 \longrightarrow X_2 \longrightarrow C(a)$ along $p: C(a) \longrightarrow X_3$ gives us the right exact sequence in $E(n)^*(-)$ we desired. The extension of this result to the short exact situation follows immediately using \cite[Lemma 7.1]{RWY}. 
\end{proof}

We now follow the proof of \cite[Theorem 1.9]{RWY}. We construct the sequence
$$\xymatrixcolsep{5pc}\xymatrix{Z \ar@{->}[r]^-{i\circ g} & Y \ar@{->}[r]^-{F:=(1, h)} & Y_+ \wedge Y_1}$$
with $F \circ (i \circ g) \simeq \ast$. This must induce a right exact sequence (of $K(m)^*$-modules) in $\tilde{K}(m)^*(-)$ for all $m>0$ by condition (3). Since $Y$ and $Y_1$ have the projective property, the right two spaces have even Morava $K$-theory and so we may apply the above lemma to obtain a right exact sequence in $\tilde{E}(n)^*(-)$. We have a completed Kunneth isomorphism (from the collapse of the Atiyah-Hirzebruch spectral sequence)
$$\tilde{E}(n)^*( Y_+ \wedge Y_1)=E(n)^*(Y) \hat{\otimes} IE(n)^*(Y_1)$$
which shows  $Z \longrightarrow Y \longrightarrow Y_1$ induces a right exact sequence of completed, graded $E(n)^*$-algebras as desired.

We now have the following diagram:
$$\xymatrix{
E(n)^{2\ast}(Y_1)\ar@{->}[r]^-{\psi_{Y_1}}_-\cong \ar@{->}[d]^-{h^*} & \mathbb{E}(n)^{\ast(1-\lambda)}(Y_1) \ar@{->}[r]^-i \ar@{->}[d]^-{h^*} & ER(n)^{\ast(1-\lambda)}(Y_1^{\Z/2}) \ar@{->}[d]^-{h^*} \\
E(n)^{2\ast}(Y)\ar@{->}[r]^-{\psi_Y}_-\cong \ar@{->}[d]^-{(i \circ g)^*} & \mathbb{E}(n)^{\ast(1-\lambda)}(Y) \ar@{->}[r]^-i \ar@{->}[d]^-{(i \circ g)^*} & ER(n)^{\ast(1-\lambda)}(Y^{\Z/2}) \ar@{->}[d]^-{(i \circ g)^*} \\
E(n)^{2\ast}(Z)\ar@{->}[r]^-{\exists \varphi_Z} & \mathbb{E}(n)^{\ast(1-\lambda)}(Z) \ar@{=}[r]  & ER(n)^{\ast(1-\lambda)}(Z) \\
}$$
Exactness (as completed, graded algebras) of the left vertical column together with condition (2) now yield the desired factorization $\varphi_Z$. It follows that $(Y, Z)$ is a Landweber flat real pair and so the conclusion of Theorem \ref{main1} applies.

\begin{remark}\label{weaker}
As pointed out by the referee, it appears to the authors that condition (3) above may be weakened by requiring the left exact sequence only for $m=n$ (rather than all $m>0$). \cite[Lemma 5.1]{RWY} implies Landweber flatness of $E(n)^*(Z)$.  Our Lemma \ref{K(n)toE(n)} uses \cite[Lemma 7.1]{RWY}, which as stated requires even Morava $K$-theory (for all $m>0$). However, it seems that even $K(n)^*(-)$ (only for the $n$ of interest) may suffice. Since the stronger assumption of even Morava $K$-theory holds for all of our examples below, we do not pursue this point further here.
\end{remark}

\subsection{Eilenberg-MacLane spaces}
For each of the spaces $Z=\EM$, $\EMF$, and $\EMf$, the spaces $Y$ and $Y_1$ will both be $\BPRm{k}{l}$ for suitable choices of integers $k$ and $l$. This proof is essentially a consequence of \cite{RWY}. We shall prove the result in detail for $\EM$ and indicate the cosmetic changes required to prove the other cases. 

\medskip
\noindent
Define the map of equivariant infinite loop spaces:
\[ \delta_{k} : \BPRm{k-1}{(2^{k}-1)(1+\alpha)+\alpha} \longrightarrow \BPR{k} \]
given by the fiber of multiplication by $v_k$:
\[ v_k : \BPR{k} \longrightarrow \BPRm{k}{2^k(1+\alpha)}. \]
Now consider the composite map $\delta$ defined as: 
\[ \delta := \delta_{2m-1} \circ \delta_{2m-2} \circ \cdots \circ \delta_1 : \BPRm{0}{2m\alpha+1} \longrightarrow \BPRm{2m-1}{(2^{2m}-1)(1+\alpha)}. \]
The space $\BPRm{0}{2m\alpha+1}$ is an equivariant model for $\EM$ with trivial $\Z/2$-action up to homotopy. The homotopy fixed point spectral sequence computing $\pi_*(\BPRm{0}{2m\alpha+1}^{h\Z/2})$, which collapses immediately since it is concentrated on a single line, thus yields that the top degree homotopy group is $\Z_{(2)}$ in degree $2m+1$. The top stage of the Postnikov tower for $\BPRm{0}{2m\alpha+1}^{h\Z/2}$ now gives a map $\EM \longrightarrow \BPRm{0}{2m\alpha+1}^{h\Z/2}$ which splits the canonical map $\BPRm{0}{2m\alpha+1}^{h\Z/2} \longrightarrow \BPRm{0}{2m\alpha+1}=\Ek(\Z, 2m+1)$ uniquely up to homotopy.

We now choose our triple $(Z, Y, Y_1)$ to be 
$$\left(\EM, \BPRm{2m-1}{(2^{2m}-1)(1+\alpha)}, \BPRm{2m-1}{2^{2m-1}(1+\alpha)}\right).$$
Note that $Y$ and $Y_1$ are complete, and thus the fixed points and homotopy fixed points are equivalent. We define the map $g: Z \longrightarrow Y^{\Z/2}$ to be the splitting constructed above followed by $\delta$ (followed by the equivalence $Y^{h\Z/2} \simeq Y^{\Z/2}$):
$$\xymatrix{\EM \ar@{->}[r] & \BPRm{0}{2m\alpha+1}^{h\Z/2} \ar@{->}[r]^-{\delta} & \BPRm{2m-1}{(2^{2m}-1)(1+\alpha)}^{h\Z/2}}$$
We choose $h: Y \longrightarrow Y_1$ to be the unstable map induced by $v_{2m-1}$-multiplication on the level of spectra. Since $v_{2m-1} \circ \delta$ is equivariantly trivial, condition (2) follows. Condition (1) follows from \cite[Appendix]{Kitch-Wil-split}. Finally, condition (3) is Proposition 1.16 in \cite{RWY}. This establishes the result for $\EM$.

\medskip
\noindent
The same argument works for $\Ek(\Z/2^q, m)$ when $q \geq 1$ and $m$ is even. Additionally, when $q=1$, the argument goes through for all $m$. We simply extend $\delta$ by the Bockstein $\delta_0: (\underline{\B\D} \langle 0 \rangle/2^q)_{m\alpha}  \longrightarrow \underline{\B\D} \langle 0 \rangle_{m\alpha +1}$:
\[ \delta_{m-1} \circ \delta_{m-2} \circ \cdots \circ \delta_1 \circ \delta_0 : (\underline{\B\D} \langle 0 \rangle/2^q)_{m\alpha} \longrightarrow \underline{\B\D}\langle m-1 \rangle_{(2^{m}-1)(1+\alpha)}. \]
We chose $Y=\BPRm{m-1}{(2^m-1)(1+\alpha)}$ and $Y_1=\BPRm{m-1}{2^{m-1}(1+\alpha)}$ and define $g$ and $h$ as before. This observation establishes all other cases proving Theorem \ref{main3} for the first family of examples.

\begin{remark}This argument fails for $\Ek(\Z/2^q, 2m+1)$ for $q>1$ because $(\underline{\B\D} \langle 0 \rangle/2^q{\, })_{(2m+1)\alpha}$ has nontrivial $\Z/2$-action and so its homotopy fixed points do not have $\Ek(\Z/2^q, 2m+1)$ as a retract. The remaining Eilenberg-MacLane spaces are not so straightforward. In fact, the results of \cite{Lor16} and \cite{KLW16a} show that the computations of $ER(n)^*(\mathbb{C}P^\infty)$ and $ER(2)^*\Ek(\Z/2^q, 1)$ both involve actually making sense of Bockstein spectral sequence differentials and do not follow from $E(n)$-cohomology by tensoring up. \end{remark}

\subsection{Connective covers of $\BOg$}
The spaces $\BUg$, $\BSUg$, and $\BUg\langle 6 \rangle$ are equivariantly equivalent to the spaces $\BPRm{1}{1+\alpha}$, $\BPRm{1}{2(1+\alpha)}$ and $\BPRm{1}{3(1+\alpha)}$, respectively. Again citing \cite{Kitch-Wil-split}, these spaces are complete and satisfy the projective property. They will be our choice of $Y$. Their fixed points admit canonical $H$-maps from the spaces  $Z=\BOg$, $\BSOg$, and $\widetilde{\BSg}$ respectively. Thus one obtains a map $g: Z \longrightarrow Y^{\Z/2}$ for each of them. In each case, we choose $Y_1=Y$ and take the self-map given by $1-c$, where $c$ denotes the generator of $\Z/2$. Our three sequences $Z\longrightarrow Y \longrightarrow Y_1$ are given as follows:
\begin{alignat*}{3}
\BOg \xymatrix{\ar@{->}[r]^-{i \circ g}&}  &\BPRm{1}{1+\alpha} &\xymatrix{\ar@{->}[r]^-{1-c}&}  &\BPRm{1}{1+\alpha}\\
\BSOg \xymatrix{\ar@{->}[r]^-{i \circ g}&}  &\BPRm{1}{2(1+\alpha)} &\xymatrix{\ar@{->}[r]^-{1-c}&} &\BPRm{1}{2(1+\alpha)}\\
\widetilde{\BSg}  \xymatrix{\ar@{->}[r]^-{i \circ g}&}  &\BPRm{1}{3(1+\alpha)} &\xymatrix{\ar@{->}[r]^-{1-c}&} &\BPRm{1}{3(1+\alpha)}
\end{alignat*}
Since $i \circ (1-c)$ is equivariantly null-homotopic, the above sequences all satisfy condition (2). For the first two sequences, condition (3) is \cite[Theorem 6.4]{KLW}. 

To prove condition (3) for $\widetilde{\BSg}$, we must unravel some results of \cite{KLW}. Forgetting the $\Z/2$-action, the third sequence above is, in the notation of \cite{KLW}, given by a sequence $ \widetilde{\BSg} \longrightarrow \underline{bu}_6 \longrightarrow \underline{bu}_6$ which may be constructed by splicing together the fibrations: 
$$ \widetilde{\BSg} \longrightarrow \underline{bu}_6 \longrightarrow \underline{\BSOg}_2, \hspace{20pt}\text{and} \hspace{20pt} \underline{\BSOg}_2 \longrightarrow \underline{bu}_6 \longrightarrow \underline{\BOg}_4.$$
The second fibration above gives rise to a short exact sequence of evenly graded bicommutative Hopf algebras in $K(m)$-homology for $m >0$ by \cite[Theorem 2.3.6]{KLW}. It remains to show that the first fibration also gives rise to such a short exact sequence. To do this we consider the following commutative diagram with all rows and columns being fibrations of infinite loop spaces: 

$$\xymatrix{
\Ek(\Z, 3) \ar@{=}[r] \ar@{->}[d] & \Ek(\Z, 3) \ar@{->}[r] \ar@{->}[d] & \ast \ar@{->}[d] \\
\widetilde{\BSg} \ar@{->}[r] \ar@{->}[d] & \underline{bu}_6 \ar@{->}[r] \ar@{->}[d] & \underline{\BSOg}_2\ar@{=}[d] \\
\BSg \ar@{->}[r] & \underline{bu}_4 \ar@{->}[r]  &  \underline{\BSOg}_2. \\
}$$

By \cite[Theorem 2.3.6]{KLW}, the Morava $K(m)$-homology of the bottom row and the middle column give rise to short exact sequences of evenly graded bicommutative Hopf algebras for $m>0$. It follows that the Morava $K(m)$-homology of $\Ek(\Z, 3)$ injects into the Morava $K(m)$-homology of $\widetilde{\BSg}$. Consequently, the left vertical fibration also gives rise to a short exact sequence of evenly graded bicommutative Hopf algebras (see \cite[Theorem 4.2 (ii)]{KLW}). Furthermore, we also notice that the Morava $K(m)$-homology of $\widetilde{\BSg}$ injects into that of $\underline{bu}_6$. As before, it follows that the Morava $K(m)$-homology of the middle row must also give rise to a short exact sequence of evenly graded bicommutative Hopf algebras for $m>0$. 

\noindent
From the above, we conclude that the spliced sequence, $\widetilde{\BSg} \longrightarrow \underline{bu}_6 \longrightarrow \underline{bu}_6$, is left exact in $K(m)$-homology for $m>0$. This proves condition (3) completing the proof of the theorem for $\widetilde{\BSg}$. 

It therefore remains to establish Theorem \ref{main3} for the case $Z = \BSg$. We first observe that the $H$-map: 
\[ (1-c) : \BSUg \longrightarrow \BSUg, \]
descends to the (non-equivariantly) homotopy trivial map on $\mbox{K}(\Z, 4)$, along the map given by the second Chern class. In particular, we have a (non-equivariant) lift of the $H$-map:
\[ \lift{(1-c)} : \BSUg \longrightarrow \BUg \langle 6 \rangle. \]
Recall that the space $\BSUg = \underline{\B\D}\langle 1 \rangle_{2(1+\alpha)}$ is homotopy equivalent to a space with the projective property. In addition, the space $\BUg \langle 6 \rangle = \underline{\B\D}\langle 1 \rangle_{3(1+\alpha)}$ is a space in the Omega spectrum for the complete $\MU$-module spectrum $\Maps(E\Z/2_+, \B\D \langle 1 \rangle)$. It follows from Theorem \ref{main1} that there is a unique equivariant lift of $(1-c) : \underline{\B\D}\langle 1 \rangle_{2(1+\alpha)} \longrightarrow \underline{\B\D}\langle 1 \rangle_{2(1+\alpha)}$:
\[
\xymatrix{
  &  \underline{\B\D}\langle 1 \rangle_{3(1+\alpha)} \ar[d] \\
 \underline{\B\D}\langle 1 \rangle_{2(1+\alpha)} \ar[ur]^{\lift{(1-c)}} \ar[r]^{(1-c)} & \underline{\B\D}\langle 1 \rangle_{2(1+\alpha)}.
}
\]

We take our sequence to be
$$\xymatrix{\BSg \ar@{->}[r] & \BPRm{1}{2(1+\alpha)} \ar@{->}[r]^{\lift{(1-c)}} & \BPRm{1}{3(1+\alpha)}}$$
As above, the right two spaces are complete and have the projective property. By \cite[Theorem 6.4]{KLW}, this sequence satisfies condition (3). It remains to show condition (2). To do this, it is sufficient to show that the following composite is null:
\[ \lift{(1-c)} \circ i : \BSg \longrightarrow \underline{\B\D}\langle 1 \rangle_{2(1+\alpha)}^{h\Z/2} \longrightarrow \underline{\B\D}\langle 1 \rangle_{3(1+\alpha)}^{h\Z/2}. \]
But notice that this is covered by 
\[ (1-c) \circ i : \BSg \longrightarrow \underline{\B\D}\langle 1 \rangle_{2(1+\alpha)}^{h\Z/2} \longrightarrow \underline{\B\D}\langle 1 \rangle_{2(1+\alpha)}^{h\Z/2}. \]
which is clearly null. It follows that the map $\lift{(1-c)} \circ i$ above factors through the space $X := \underline{\B\D}\langle 0 \rangle_{1+2\alpha}^{h\Z/2}$. As above, the homotopy fixed point spectral sequence computing $\pi_*X$ shows that the top homotopy group of $X$ is in degree 3. It follows that any map from $\BSg$ to the space $X$ is null. This completes the proof of Theorem \ref{main3}.

\subsection{Short exact sequences} 

Let us now establish the existence of the short exact sequences of completed algebras claimed in Theorem \ref{main4}. We will make use of the following lemma.
\begin{lemma} Suppose the sequence $\xymatrix{Z_3 \ar@{->}[r]^-a & Z_2 \ar@{->}[r]^-b & Z_1}$ induces a short exact sequence in $E(n)$-cohomology, $Z_1$ and $Z_2$ are part of Landweber flat real pairs, and $E(n)^*(Z_3)$ is Landweber flat. Then it also induces a short exact sequence in $ER(n)$-cohomology. 
\end{lemma}
\begin{proof}The short exact sequence in $E(n)^*(-)$ implies a short exact sequence on the $E_1$-page of the BSS. For each $i$, let $\hat{E}(n)^*(Z_i)$ denote the hatted subalgebra of $E(n)^*(Z_i)$ as above. By Theorem \ref{main2} we have isomorphisms of spectral sequences $E_r^{*, *}(Z_i)=E_r^{*, *}(pt) \otimes_{\hat{E}(n)^*} \hat{E}(n)^*(Z_i)$ for $i=1, 2$. For $i=3$, note that we at least have this isomorphism on the $E_1$-page of the BSS, $E_1^{*, *}(Z_3)=E_1^{*, *}(pt) \otimes_{\hat{E}(n)^*} \hat{E}(n)^*(Z_3)$. Since $a^*$ is a surjection, $\hat{E}(n)^*(Z_2)$, which consists of permanent cycles, maps onto $\hat{E}(n)^*(Z_3)$. Furthermore, since $\hat{E}(n)^*(Z_3)$ is Landweber flat, it follows that we in fact have an isomorphism of spectral sequences, $E_r^{*, *}(Z_3)=E_r^{*, *}(pt) \otimes_{\hat{E}(n)^*} \hat{E}(n)^*(Z_3)$. Since $a^*$ is a surjection on $\hat{E}(n)^*(-)$, it follows from the previous statement that it is a surjection on each page of the BSS. Thus, we have a short exact sequence of spectral sequences
$$\footnotesize \xymatrix{0 \ar@{->}[r] & E_r^{*, *}(pt) \otimes_{\hat{E}(n)^*} \hat{E}(n)^*(Z_1) \ar@{->}[r]^-{1 \otimes b^*} & E_r^{*, *}(pt) \otimes_{\hat{E}(n)^*} \hat{E}(n)^*(Z_2)  \ar@{->}[r]^-{1 \otimes a^*} & E_r^{*, *}(pt) \otimes_{\hat{E}(n)^*} \hat{E}(n)^*(Z_3) \ar@{->}[r] & 0}$$
In particular, we have a short exact sequence on the $E_\infty$-page. Using the Snake Lemma and induction on filtrations (which are finite), we notice that this must represent a short exact sequence of $ER(n)^*$-modules
$$0 \longrightarrow ER(n)^*(Z_1) \longrightarrow ER(n)^*(Z_2) \longrightarrow ER(n)^*(Z_3) \longrightarrow 0$$
\end{proof}
Without loss of generality, let us establish the existence of the short exact sequence of completed algebras:
\[ 1 \lra ER(n)^*(\Ek(\Z/2,3)) \lra ER(n)^*(\widetilde{\BSg}) \lra ER(n)^*(\BOg\langle 8 \rangle) \lra 1. \]
The remaining cases will follow exactly along the same lines. First notice that the proof of the previous theorem (or \cite{KLW}) shows that we have the following short exact sequence of $K(n)$-algebras:
\[1 \lra K(n)^*(\Ek(\Z/2,3)) \lra K(n)^*(\widetilde{\BSg}) \lra K(n)^*(\BOg\langle 8 \rangle) \lra 1. \]
Following the lines of the proof of \cite[Theorem 1.19]{RWY}, this gives rise to a short exact sequence of (reduced) $K(n)^*$-modules:
\[ 0 \lra \tilde{K}(n)^*(\widetilde{\BSg}_+ \wedge \Ek(\Z/2,3)) \lra \tilde{K}(n)^*(\widetilde{\BSg}_+) \lra \tilde{K}(n)^*(\BOg\langle 8 \rangle_+) \lra 0, \]
where the map $\tilde{K}(n)^*(\widetilde{\BSg}_+ \wedge \Ek(\Z/2,3)) \lra \tilde{K}(n)^*(\widetilde{\BSg}_+)$ is induced by the diagonal map on $\widetilde{\BSg}$ followed by the bottom cohomology class on the second factor. We may use the fact that all spaces in the short exact sequence above have evenly graded $K(n)$-cohomology and invoke Lemma \ref{K(n)toE(n)} to see that one has a short exact sequence in (reduced) $E(n)^*$-modules: 
\[ 0 \lra \tilde{E}(n)^*(\widetilde{\BSg}_+ \wedge \Ek(\Z/2,3)) \lra \tilde{E}(n)^*(\widetilde{\BSg}_+) \lra \tilde{E}(n)^*(\BOg\langle 8 \rangle_+) \lra 0. \]
The above lemma now implies that we have a short exact sequence in $ER(n)^*(-)$. Applying the Kunneth isomorphism of Theorem \ref{Kunneth} completes the proof.

\begin{remark}\label{chooseY} Note that the above argument does not use functoriality in the spaces $Y_i$ corresponding to the $Z_i$. In fact, we do not require $Z_3$ to even be part of a Landweber flat real pair (and we do not know of such a space in the case of $\BOg\langle 8 \rangle$ for $n>2$). A choice of $Y_i$ gives a lift of the subalgebra $\hat{E}(n)^*(Z_i)$ of $E(n)^*(Z_i)$ to a subalgebra of $ER(n)^*(Z_i)$ of the same name. In the case of $Z_3$, we do not have such a canonically defined lift, and we do \emph{not} know of an isomorphism $ER(n)^*(Z_3)=ER(n)^*\otimes_{\hat{E}(n)^*}\hat{E}(n)^*(Z_3)$. Also, nothing in our proof requires the hatted subalgebra of $ER(n)^*(Z_1)$ to map to the hatted subalgebra of $ER(n)^*(Z_2)$. We think of choosing a space with projective property $Y$ corresponding to a space $Z$ as akin to choosing multiplicative generators. Our construction of the short exact sequences is coordinate free in $ER(n)^*(-)$.
\end{remark}

\pagestyle{empty}

\begin{thebibliography}{RWY98}

\bibitem[And64]{AndersonKOBO}
D.~W. Anderson.
\newblock The real {K}-theory of classifying spaces.
\newblock {\em Proceedings of the National Academy of Sciences U.S.A.},
  51(4):634--636, 1964.
  
\bibitem[AM78]{AM78}
S. ~Araki and M. ~Murayama
\newblock $\tau$-cohomology theories
\newblock {\em Japan J. Math. (N.S.)}, 
 4(2): 363-416, 1978
 
\bibitem[Ara79a]{Ara79a}
S. ~Araki
\newblock Forgetful spectral sequences
\newblock {\em Osaka J. Math.}, 
 16(1): 173-199, 1979
 
\bibitem[Ara79b]{Ara79b}
S. ~Araki
\newblock Orientations in {$\tau $}-cohomology theories
\newblock {\em Japan. J. Math. (N.S.)}, 
 5(2): 403--430, 1979

\bibitem[AS69]{AS}
M.~F. Atiyah and G.~B. Segal.
\newblock Equivariant {K}-theory and completion.
\newblock {\em J. Differential Geometry}, 3:1--18, 1969.

\bibitem[Ban13]{Ban}
R. ~Banerjee.
\newblock On the {$ER(2)$}-cohomology of some odd-dimensional projective spaces.
\newblock {\em Topology and its applications}, 160(12):1395-1405, 2013


\bibitem[Har91]{HaraK}
S.~Hara.
\newblock Note on {KO}-theory of {BO}(n) and {BU}(n).
\newblock {\em Journal of Mathematics of Kyoto University}, 31(2):487--493,
  1991.
  
  \bibitem[HH95]{HH}
M. ~J. Hopkins and J. ~R. Hunton
\newblock On the structure of spaces representing a {L}andweber exact cohomology theory.
\newblock {\em Topology}, 34(1):29--36,
  1995.
  
\bibitem[HM17]{HM}
M.~A. Hill and L.~Meier.
\newblock The ${C}_2$--spectrum ${T}mf_1(3)$ and its invertible modules.
\newblock {\em Algebraic \& Geometric Topology}, 17(4):1953-2011, 2017.

\bibitem[HK01]{HK}
P.~Hu and I.~Kriz.
\newblock Real-oriented homotopy theory and an analogue of the
  {Adams}-{Novikov} spectral sequence.
\newblock {\em Topology}, 40(2):317--399, 2001.

\bibitem[HS05]{HS}
M.~Hovey and N.~Strickland.
\newblock Comodules and {L}andweber exact homology theories.
\newblock {\em Advances in Mathematics}, 192:427--456, 2005.

\bibitem[Kri97]{Kriz}
I.~Kriz.
\newblock Morava {$K$}-theory of classifying spaces: some calculations.
\newblock {\em Topology}, 36(6):1247--1273, 1997.

\bibitem[KS04]{KS}
I.~Kriz and H.~Sati.
\newblock M-theory, type iia superstrings, and elliptic cohomology.
\newblock {\em Adv. Theor. Math. Phys.}, 8:345--395, 2004.

\bibitem[KLW04]{KLW}
N.~Kitchloo, G.~Laures and W.~S. Wilson.
\newblock The Morava K-theory of spaces related to {BO}. 
\newblock Advances in Math., 189(1): 192--236, 2004. 

\bibitem[KLW16a]{KLW16a}
N.~Kitchloo, V.~Lorman and W.~S. Wilson.
\newblock The {$ER(2)$}-cohomology of {$B\Z/2^q$} and {$\mathbb{C}P^n$}.
\newblock {\em Canadian Journal of Mathematics}, 2017. 

\bibitem[KLW16b]{KLW16b}
N.~Kitchloo, V.~Lorman and W.~S. Wilson.
\newblock Multiplicative structures on Real Johnson-Wilson theory.
\newblock {\em Contemporary Mathematics, New Directions in Homotopy Theory} (to appear), 2017.

\bibitem[KW07a]{Nitufib}
N.~Kitchloo and W.~S. Wilson.
\newblock On fibrations related to real spectra.
\newblock In M.~Ando, N.~Minami, J.~Morava, and W.~S. Wilson, editors, {\em
  Proceedings of the Nishida Fest (Kinosaki 2003)}, volume~10 of {\em Geometry
  \& {T}opology Monographs}, pages 237--244, 2007.

\bibitem[KW07b]{NituER2}
N.~Kitchloo and W.~S. Wilson.
\newblock On the {Hopf} ring for {$ER(n)$}.
\newblock {\em Topology and its {A}pplications}, 154:1608--1640, 2007.

\bibitem[KW08a]{NituP}
N.~Kitchloo and W.~S. Wilson.
\newblock The second real {J}ohnson-{W}ilson theory and non-immersions of
  ${RP}^n$.
\newblock {\em Homology, Homotopy and Applications}, 10(3):223--268, 2008.

\bibitem[KW08b]{NituP2}
N.~Kitchloo and W.~S. Wilson.
\newblock The second real {J}ohnson-{W}ilson theory and non-immersions of
  ${RP}^n$, {P}art 2.
\newblock {\em Homology, Homotopy and Applications}, 10(3):269--290, 2008.

\bibitem[KW13]{Kitch-Wil-split}
N.~Kitchloo and W.S. Wilson.
\newblock Unstable splittings for real spectra.
\newblock {\em Algebraic and Geometric Topology}, 13(2):1053--1070, 2013.

\bibitem[KW15a]{Kitch-Wil-KnBOq}
N.~Kitchloo and W.S. Wilson.
\newblock The {M}orava {K}-theory of {BO(q)} and {MO(q)}.
\newblock {\em Algebraic and Geometric Topology}, 15(5):3049--3058, 2015.

\bibitem[KW15b]{KW2}
N.~Kitchloo and W.S. Wilson.
\newblock The {$ER(n)$}-cohomology of {$BO(q)$} and real Johnson-Wilson orientations for vector bundles
\newblock {\em Bulletin of the London Math. Soc.}, 47(5):835--847, 2015.

\bibitem[KY93]{KY}
A.~Kono and N.~Yagita.
\newblock {Brown-Peterson} and ordinary cohomology theories of classifying
  spaces for compact {Lie} groups.
\newblock {\em Transactions of the American Mathematical Society},
  339(2):781--798, 1993.

\bibitem[Lan68]{Lan68}
P.~S. Landweber
\newblock Conjugations on complex manifolds and equivariant homotopy of ${MU}$
\newblock {\em Bulletin of the American Mathematical Society}, 74:271-274, 1968

\bibitem[Lan73a]{Land:Ann}
P.~S. Landweber.
\newblock Annihilator ideals and primitive elements in complex cobordism.
\newblock {\em Illinois Journal of Mathematics}, 17:273--284, 1973.

\bibitem[Lan73b]{Land:Ass}
P.~S. Landweber.
\newblock Associated prime ideals and {Hopf} algebras.
\newblock {\em Journal of Pure and Applied Algebra}, 3:175--179, 1973.

\bibitem[Lan76]{Land:Hom}
P.~S. Landweber.
\newblock Homological properties of comodules over {$MU_*(MU)$} and
  {$BP_*(BP)$}.
\newblock {\em American Journal of Mathematics}, 98:591--610, 1976.

\bibitem[LMS]{LMS}
Lewis, Jr., L. G. and May, J. P. and Steinberger, M. and McClure, J. E.
\newblock Equivariant stable homotopy theory, volume 1213 of {\em Lecture Notes in Mathematics}. 
\newblock Springer-Verlag, Berlin, 1986.
\newblock With contributions by J. E. McClure.


\bibitem[LO]{LO}
G.~Laures and M.~Olbermann.
\newblock {\em {$TMF_0(3)$}-characteristic classes for string bundles}. 
\newblock {\em Mathematische Zeitschrift}, 282(1-2):511--533, 2016.

\bibitem[Lor16]{Lor16}
V. ~Lorman
\newblock {The real {J}ohnson-{W}ilson cohomology of {$\Bbb{CP}^\infty$}}
\newblock {\em Topology and its Applications}, 209: 367-388, 2016.
		
\bibitem[RW80]{RW}
D.~C. Ravenel and W.~S. Wilson
\newblock The {M}orava {$K$}-theories of {E}ilenberg-{M}ac {L}ane spaces and the {C}onner-{F}loyd conjecture
\newblock {\em Amer. J. Math.}, 102(4):691--748, 1980.

\bibitem[RW76]{RW2}
D.~C. Ravenel and W.~S. Wilson
\newblock The {H}opf ring for complex cobordism
\newblock {\em. Pure Appl. Algebra}, 9(3):241--280, 1976/77.
	


\bibitem[RWY98]{RWY}
D.~C. Ravenel, W.~S. Wilson, and N.~Yagita.
\newblock {Brown-Peterson} cohomology from {Morava} {$K$}-theory.
\newblock {\em ${K}$-Theory}, 15(2):149--199, 1998.

\bibitem[Wil84]{WSW:BO}
W.~S. Wilson.
\newblock The complex cobordism of {$BO_n$}.
\newblock {\em Journal of the London Mathematical Society}, 29(2):352--366,
  1984.

\bibitem[Wil99]{W99}
W.~S. Wilson.
\newblock Brown-Peterson cohomology from Morava K-theory, II.
\newblock {\em K-theory}, 17:95--101, 1999.

\end{thebibliography}

\end{document}